\documentclass[12pt,reqno]{amsart}
\usepackage[totalwidth=440pt,totalheight=640pt,centering]{geometry}
\usepackage[utf8]{inputenc}
\usepackage[T1]{fontenc}
\DeclareUnicodeCharacter{012D}{\u{\i}}
\usepackage{amsmath,amssymb,amsthm,mathtools,pdflscape}
\usepackage[all]{xy}
\mathtoolsset{showonlyrefs}
\usepackage[mathcal]{euscript}
\usepackage{mathrsfs}
\let\mathcal=\mathscr
\usepackage{microtype}
\usepackage{xcolor}
\usepackage[colorlinks=true,linkcolor=teal,citecolor=purple,urlcolor=cyan]{hyperref}
\usepackage{comment}

\newtheorem{thm}{Theorem}[section]
\newtheorem{corol}[thm]{Corollary}
\newtheorem{lemma}[thm]{Lemma}
\newtheorem{prop}[thm]{Proposition}
\newtheorem{defin}[thm]{Definition}

\theoremstyle{remark}
\newtheorem{rem}[thm]{Remark}
\newtheorem{ex}[thm]{Example}
\newenvironment{remark}{\begin{rem}\rm}{\qee\end{rem}}

\newcommand{\cA}{{\mathcal A}}

\newcommand{\cC}{{\mathcal C}}
\newcommand{\calD}{{\mathcal D}}
\newcommand{\cE}{{\mathcal E}}
\newcommand{\cF}{{\mathcal F}}
\newcommand{\cO}{{\mathcal O}}

\newcommand{\calL}{{\mathcal L}}
\newcommand{\cM}{{\mathcal M}}

\newcommand{\cI}{{\mathcal I}}
\newcommand{\cJ}{{\mathcal J}}
\newcommand{\cK}{{\mathcal K}}
\newcommand{\cQ}{{\mathcal Q}}
\newcommand{\cS}{{\mathcal S}}
\newcommand{\cT}{{\mathcal T}}
\newcommand{\cZ}{{\mathcal Z}}

\newcommand{\End}{\operatorname{End}}
\newcommand{\Ext}{\operatorname{Ext}}

\newcommand{\Lie}{\mathscr{L}\!\mathit{ie}}

\newcommand{\Kernel}{\operatorname{Ker}}
\newcommand{\Hom}{\operatorname{Hom}}
\newcommand{\gr}{\operatorname{gr}}

\newcommand{\Rep}{\operatorname{Rep}}

\newcommand{\N}{{\mathbb N}}

\newcommand{\qee}{\mbox{\hspace{0.2mm}}\hfill$\triangle$}

\newcommand{\Der}{\operatorname{Der}}
\newcommand{\Dder}{\operatorname{\mathcal{D}\mspace{-2mu}\mathit{er}}}

\newcommand{\cHom}{\operatorname{{\mathcal{H}\mspace{-2mu}\mathit{om}}}}

\newcommand{\Out}{\operatorname{{\cO}\!\mathit{ut}}}

\newcommand{\Vect}{\mathbf{Vect}_\Bbbk}

\begin{document}
%
%
%
%
%

 \begin{center}
{\Large\bf Lie algebroid cohomology and \\[5pt] Lie algebroid extensions } \\[20pt]
{\sc E. Aldrovandi,$^\P$ U. Bruzzo$^{\S\star\dag}$ and V. Rubtsov$^{\ddag\sharp}$} \\[10pt]
$^\P$Department of Mathematics, Florida State University,  \\ 1017 Academic Way, Tallahassee FL  32304, USA  \\[5pt]
$^\S$Scuola Internazionale Superiore di Studi Avanzati, \\ Via Bonomea 265, 34136 Trieste, Italy  \\[5pt]
 $^\star$Istituto Nazionale di Fisica Nucleare, Sezione di Trieste\\[5pt]
 $\dag$ Arnold-Regge Institute, Torino \\[5pt]
$^\ddag$ Universit\'e d'Angers, D\'epartement de Math\'ematiques, \\ UFR Sciences,   LAREMA, UMR 6093 du CNRS,
2 bd.~Lavoisier, \\ 49045 Angers Cedex 01, France  \\[5pt]
$^\sharp$ ITEP Theoretical Division, 25 Bol.~Tcheremushkinskaya, \\  117259, Moscow, Russia 
\end{center} 

\bigskip {\small
\begin{quote}{\sc Abstract.} We consider the extension problem for Lie algebroids over schemes over a field. Given a 
 locally free Lie algebroid  $\cQ$ over a scheme $X$, and  a sheaf of  finitely generated Lie $\cO_X$-algebras  $\calL$,
 we determine the obstruction to the existence of    extensions 
$0 \to \calL \to \cE \to \cQ \to 0 $, and classify the extensions in terms of a suitable Lie algebroid hypercohomology group.  In the preliminary sections we study free Lie algebroids and recall some basic facts about Lie algebroid hypercohomology.
\end{quote}
}

\bigskip
 
 \let\svthefootnote\thefootnote
\let\thefootnote\relax\footnote{
\hskip-\parindent {\em Date: } \today  \\
{\em 2000 Mathematics Subject Classification:} 14F40, 18G40, 32L10,   55N35, 55T05 \\ 
{\em Keywords:} Lie algebroid extensions, Lie algebroid cohomology, free Lie algebroids \\
Email: {\tt  ealdrov@math.fsu.edu, bruzzo@sissa.it, Volodya.Roubtsov@univ-angers.fr}
}
\addtocounter{footnote}{-1}\let\thefootnote\svthefootnote
\newpage
\setcounter{tocdepth}{1}
{\small\tableofcontents}

\section{Introduction}

Let $X$ be a noetherian separated scheme over a field $\Bbbk$.  Let $\cQ$ be a locally free Lie algebroid  on $X$, and  $\calL$ a sheaf of  finitely generated Lie $\cO_X$-algebras   (definitions will be given in Section \ref{cohom}).
An exact sequence of Lie algebroids
\begin{equation}\label{extalg} 0 \to \calL \to \cE \to \cQ \to 0 \end{equation}
is called an extension of $\cQ$ by $\calL$. Any such extension defines a morphism   $\alpha\colon\cQ\to\Out(\calL)$,
where $\Out(\calL)$ is the Lie algebroid of outer derivations of $\calL$, by letting
$$ \alpha(x) (y) = \{x',y\}_\cE$$
where $x'$ is any counterimage of $x$ in $\cE$. 
 It also induces a representation
of $\cQ$ on the centre $Z(\calL)$ of $\calL$, i.e., a morphism $\alpha\colon\cQ\to\Dder(Z(\calL))$.

Any two extensions $\cE_1$, $\cE_2$ are considered to be equivalent if there is a morphism $\cE_1\to \cE_2$
such that the diagram 
\begin{equation} \label{equiv}
\xymatrix{
0 \ar[r] & \calL  \ar[r]  \ar@{=}[d]  & \cE _1 \ar[r]  \ar[d] &\cQ  \ar[r]  \ar@{=}[d] & 0 \\
0  \ar[r]  & \calL \ar[r]  &\cE_2 \ar[r] & \cQ  \ar[r] & 0
}
\end{equation}
commutes.

 In this paper we study the problem of finding extensions of Lie algebroids  as in \eqref{extalg}
 such that the induced $\cQ$-module structure of $Z(\calL)$ coincides with that defined by a given   $\alpha$.  
This problem was already studied in \cite{BMRT} by realizing the hypercohomology groups of a Lie algebroid
in terms of \v Cech complexes. Here, following \cite{Ugo-derived}, we adopt an intrinsic approach.

If $\calL$ is abelian, the problem is unobstructed, as $\alpha$ defines an action of $\cQ$ on $\calL$, and  
one can   define the semidirect product Lie algebroid
$$ \cE = \calL \rtimes_\alpha \cQ, $$
where $\cE=\calL\oplus\cQ$ as $\cO_X$-modules,
with bracket
$$\{ (\ell,x),\,(\ell',x')\} = (\alpha(x)(\ell')-\alpha(x')(\ell), \{x,x'\})$$
and anchor $a\colon\cE\to\Dder(\cA)$ given by the anchor $b$ of $\cQ$, i.e., $a((\ell,x))=b(x).$
On the other hand, if $\calL$ is not abelian, $\alpha$ does not define an action of $\cQ$ on $\calL$, and the problem of finding an extension of $\cQ$ by $\calL$ is obstructed by a class $\mathbf{ob}(\alpha)$ in the group
$$ \mathbb H^3 (\cQ;Z(\calL))^{(1)} = \mathbb H^3(X,\sigma^{\ge 1}Z(\calL)\otimes\Lambda^\bullet \cQ^\ast),$$
i.e., the third hypercohomology group  of a ``sharp'' truncation of the Chevalley-Eilenberg-de Rham complex of $\cQ$ with coefficients in $Z(\calL)$.

When the obstruction is zero (which, as we have seen, is always the case when $\calL$ is abelian), the equivalence classes
of extensions of $\cQ$ by $\calL$, inducing on $Z(\calL)$ the $\cQ$-action given by $\alpha$, are classified by the group
$$ \mathbb H^2 (\cQ;Z(\calL))^{(1)} = \mathbb H^2(X,\sigma^{\ge 1}\calL\otimes\Lambda^\bullet \cQ^\ast).$$
So we have the following theorem.

\begin{thm} Given a locally free Lie algebroid $\cQ$,   a sheaf $\calL$  of finitely generated Lie $\cO_X$-algebras, and a morphism $\alpha\colon\cQ\to \Out(\calL)$, the problem of finding an extension of $\cQ$ by $\calL$ 
 inducing on $Z(\calL)$ the $\cQ$-action given by $\alpha$
is obstructed by a class $\mathbf{ob}(\alpha)\in  \mathbb H^3 (\cQ;Z(\calL))^{(1)}$. 
If $\mathbf{ob}(\alpha)=0$, the space of equivalence classes
of extensions is a torsor over the group $\mathbb H^2 (\cQ;Z(\calL))^{(1)}$.
\label{class}
\end{thm} 

\begin{remark} In the abelian case the  space of equivalence classes
of extensions is naturally identified with $\mathbb H^2 (\cQ;Z(\calL))^{(1)}$,
with the zero element of the latter space being identified with the semidirect product extension.
\end{remark}

The contents of this paper are as follows. In Section \ref{cohom} we review the fundamentals about Lie algebroid cohomology, stressing a few facts that will be needed later on in the paper. Since some arguments will involve the use of free Lie algebroids, in Section \ref{freealg} we develop their basic theory. In Section \ref{abelian} we briefly 
treat the abelian case; the classification problem in this case fits into a general theory developed by van Osdol \cite{vanosdol}.
In Section \ref{nonabelian} we treat the nonabelian case. We first construct the obstruction to the extension problem for Lie algebroids, and then, assuming that the obstruction vanishes, we reduce the classification theorem to the abelian case.

\smallskip {\noindent \bf Acknowledgements} Part of this research was  carried out   while  U.B~was visiting the Institute for Mathematical Sciences, National University of Singapore in 2016, and the Department of Mathematics of the University of Pennsylvania in 2017.  We thank G.~Powell and F.~Sorrentino for useful discussions. This research was partly supported by  INdAM-GNSAGA (also through its Visiting Professor Programme). U.B.~is a member of VBAC. V.R's research is partly supported by RFBR's grant 15-01-05990A. 

\bigskip
\section{Lie algebroids and their hypercohomology} \label{cohom}
In this section, basically following \cite{Rinehart63} and \cite{Ugo-derived}, 
we recall some basic facts about the cohomology of Lie-Rinehart algebras, and the hypercohomology
of Lie algebroids over schemes.

\subsection{Lie-Rinehart algebras}  \label{BenjiVolodya}  As Lie algebroids  are in a sense Lie-Rinehart algebras with coefficients, we start with some issues about Lie-Rinehart algebras.

Let $A$ be a finitely generated commutative, associative unital algebra over a field $\Bbbk$. A $(\Bbbk,A)$-Lie-Rinehart algebra is a pair $(L,a)$, where $L$ is an $A$-module equipped with a $\Bbbk$-linear Lie algebra bracket $\{\,,\,\}$, and $a\colon L \to \operatorname{Der}_\Bbbk(A)$ a representation of $L$ in $\operatorname{Der}_\Bbbk(A)$  (the anchor) that satisfies the Leibniz rule 
$$ \{s,ft\} = f\{s,t\} + a(s)(f)\,t $$
where   $s,t\in L$ and $f\in A$. 

We consider a useful class of  Lie-Rinehart algebras.
Let $\Bbbk$ be a field, and $A$ a commutative associative algebra over $\Bbbk$.
Let $V$ be a $\Bbbk$-vector space, and define
$$L=A \otimes_\Bbbk V.$$
Let $G^\bullet(L)$ be the graded algebra generated by $L$ over $A$, with
$A$ in degree 0 and $L$ in degree one. We have \cite{EnriRub}

\begin{prop} Lie-Rinehart algebra structures on $L$ are in a one-to-one correspondence
with degree -1, graded-symmetric $\Bbbk$-Lie brackets on 
$$G^0(L) \oplus G^1(L) =  A \oplus L $$
that satisfy the Leibniz rule
$$ [ \alpha s,s'] = \alpha [s,s'] + [\alpha,s'] s $$
for $\alpha\in A$, $s,s'\in L$.
\end{prop} 
Note that the bracket is required to satisfy a graded Jacobi identity, which implies the usual Jacobi identity for $L$. Moreover, the anchor of $L$ is given by the map
$$a\colon L\to \Der_\Bbbk(A),\qquad a(x)(\alpha) = [\alpha,x].$$

A useful class of examples is provided by taking $V=\mathfrak g$, where $\mathfrak g$ is a Lie
algebra over $\Bbbk$ equipped with a $\Bbbk$-Lie algebra homomorphism $a\colon \mathfrak g \to \Der_\Bbbk(A)$
(the Lie-Rinehart algebras obtained in this way are called {\em transformation Lie-Rinehart algebras} \cite{MK87}).  
The bracket on $G^0(L) \oplus G^1(L)$ is defined as
$$ [\alpha\otimes\xi,\beta] = \alpha\, a(\xi)(\beta)$$
$$[\alpha\otimes\xi,\beta\otimes \eta ] = 
\alpha\beta\, [\xi,\eta] + \alpha\,a(\xi)(\beta)\,\eta - \beta\,a(\eta)(\alpha)\,\xi$$
for $\alpha,\beta\in A$, $\xi,\eta\in\mathfrak g$. (Note that the bracket of two elements in $G^0(L)$ is always zero as the bracket
is supposed to have degree $-1$.) The anchor $a_L$ of $L$ is given by
$$a_L(\alpha\otimes \xi)(\beta) =\alpha\,a(\xi)(\beta).$$

\subsection{Lie algebroid cohomology} 
We consider now Lie algebroids.  All schemes will be noetherian.
Let $X$ be a  separated 
scheme over a field $\Bbbk$ (the same results hold in the holomorphic category). We shall denote by $\cO_X$ the sheaf of regular functions on $X$, by $\Bbbk_X$ the constant sheaf on $X$ with stalk $\Bbbk$, and by $\Theta_X$ the tangent sheaf of $X$ (the sheaf of derivations of the structure sheaf $\cO_X$), which is a sheaf of $\Bbbk_X$-Lie algebras.  
A   Lie algebroid $\cC$ on $X$ is a coherent $\cO_X$-module $\cC$ equipped with:
\begin{itemize}
\item a $\Bbbk$-linear Lie bracket defined on   sections of $ \cC$, satisfying the Jacobi identity;
\item a   morphism of $\cO_X$-modules $a\colon \cC \to\Theta_X$, 
called the {\em anchor} of $\cC$, which is also a morphism of sheaves of $\Bbbk$-Lie algebras.
\end{itemize}
The  Leibniz rule
\begin{equation}\label{leibniz} \{s,ft\} = f\{s,t\} +a(s)(f)\,t  \end{equation}
is required to hold 
for all sections $s,t$ of $\cC$ and $f$ of $\cO_X$ (actually the Leibniz rule and the Jacobi identity imply that the anchor is a morphism of $\Bbbk_X$-Lie algebras).

A morphism $(\cC,a)\to (\cC',a')$ of Lie algebroids defined over the same scheme $X$ is a morphism of $\cO_X$-modules $f\colon\cC\to\cC'$, which is compatible with the brackets defined in $\cC$ and in $\cC'$, and such that $a'\circ f=a$.  Note that this implies
that the kernel of a morphism of Lie algebroids has a trivial anchor, i.e., it is a sheaf of $\cO_X$-Lie algebras.

\begin{defin} A representation of a Lie algebroid $\cC$ is a pair $(\cM,\rho)$, where $\cM$ is a coherent $\cO_X$-module,
and $\rho$ is an $\cO_X$-linear morphism $\cC\to\End_\Bbbk(\cM)$ satisfying the condition
$$\rho(x)(fm) = f\rho(x)(m) +a(x)(f)m$$
for all sections $f$, $x$ and $m$ of $\cO_X$, $\cC$ and $\cM$, respectively.
\end{defin}
A representation $\cM$ of $\cC$ will also be called a $\cC$-module.
We shall denote by $\Rep(\cC)$ the category of representations of a Lie algebroid $\cC$. Given a representation $(\cM,\rho)$,
we define the {\em invariant submodule} $\cM^\cC$ of $\cM$ as the sheaf of $\Bbbk_X$-modules
$$\cM^\cC(U) = \{ m \in \cM(U) \,\vert\, \rho(\cC) (m) = 0  \}.$$
This is an $\cO_X$-module when the anchor of $\cC$ is trivial. In general, this defines a functor
$$(-)^\cC\colon\Rep(\cC)\to \Bbbk_X\hbox{\bf-mod}.$$

Assuming that $\cC$ is locally free, we introduce the {\em Chevalley-Eilenberg-de Rham complex}  of $\cC$ with coefficients in a representation $(\cM,\rho)$, which is a  sheaf of
differential graded algebras. This is  $\cM\otimes_{\cO_X}\Lambda^\bullet_{\cO_X}\cC^\ast$ as a sheaf of $\cO_X$-modules, with a product  given by the wedge product, and  a $\Bbbk$-linear differential  $d_\cC\colon \cM\otimes_{\cO_X} \Lambda^\bullet_{\cO_X}\cC^\ast\to \cM\otimes_{\cO_X}\Lambda^{\bullet+1}_{\cO_X}\cC^\ast$   defined by the  formula
  \begin{eqnarray*}\label{diff}
(d_\cC\xi)(s_1,\dots,s_{p+1}) &=& 
\sum_{i=1}^{p+1}(-1)^{i-1}\rho(s_i)(\xi(s_1,\dots,\hat s_i,
\dots,s_{p+1})) \\ & + & \sum_{i<j}(-1)^{i+j}
\xi([s_i,s_j],\dots,\hat s_i,\dots,\hat s_j,\dots,s_{p+1})
\end{eqnarray*}
  for   $s_1,\dots,s_{p+1}$ sections of $\cC$, and $\xi$ a section of $\cM\otimes_{\cO_X}\Lambda^p_{\cO_X}\cC^\ast$.
   
  The hypercohomology 
of the complex $(\cM\otimes_{\cO_X}\Lambda^\bullet_{\cO_X}\cC^\ast,d_\cC)$, denoted $\mathbb H^\bullet(\cC;\cM)$, 
 is called the {\em   Lie algebroid cohomology} of $\cC$
with coefficients in  $(\cM,\rho)$. If $X$ is affine, the hypercohomology $\mathbb H^\bullet(\cC;\cM)$
reduces the cohomology of the $(\Bbbk,\cO_X(X))$-Lie-Rinehart algebra $\cC(X)$ with coefficients in $\cM(X)$
\cite{Rinehart63}.

\subsection{Cohomology of transformation Lie algebroids}\footnote{This section is not used elsewhere in this paper. We record it here for the sake of completeness.} If $\calL$ is a locally free sheaf of $\Bbbk_X$-Lie algebras, and
$b \colon \calL \to \Dder_\Bbbk(\cO_X)$ is a morphism of Lie $\Bbbk_X$-algebras, one can, in analogy with the case
of Lie-Rinehart algebras, define the {\em transformation Lie algebroid} $\cC = \cO_X\otimes_\Bbbk \calL$, with
anchor $a(f\otimes\xi) = f\otimes b(\xi)$. Let $\cM$ be a representation of $\cC$ which is locally free as an $\cO_X$-module;
then $\cM$ is also a representation of $\calL$, and each fibre $\cM_x$ is a representation of the Lie algebra $\calL_x$
(the fibre of $\calL$ at $x\in X$. Then 
 one immediately has an isomorphism of $\Bbbk$-vector spaces 
$$\mathbb H^\bullet(\cC;\cM) \simeq \mathbb H^\bullet(\calL;\cM).$$
Moreover there is a spectral sequence converging to these groups whose second page is
$$E_2^{p,q} = H^p(X,\mathcal H^q(\calL;\cM)).$$
Here $ \mathcal H^q(\calL;\cM)$ is a vector bundle whose fibre at $x\in X$ is the Chevalley-Eilenberg cohomology
$H^q(\calL_x;\cM_x)$ of the Lie algebra $\calL_x$ with coefficients in the vector space $\cM_x$.

\subsection{Lie algebroid cohomology as a derived functor}
\label{sec:lie-algebr-cohom}
Given a locally free algebroid $\cC$, we consider the functor
$$I^\cC\colon\Rep(\cC)\to \Bbbk\hbox{\bf-mod},\qquad \cM \mapsto \Gamma(X,\cM^\cC).$$
This is left-exact, and since $\Rep(\cC)$ has enough injectives \cite{Ugo-derived}, we can take its derived functors. 
It was shown in \cite{Ugo-derived} that these derived functors are isomorphic to the hypecohomology functors, that is, for
every representation $\cM$ of $\cC$ there are functorial isomorphisms
\[ R^iI^\cC( \cM) \simeq \mathbb  H^i(\cC;\cM), \quad i\ge 0.\]
In the same way, the derived functors of the functor $(-)^\cC$ applied to a representation $\cM$ give the cohomology sheaves
of the Chevalley-Eilenberg-de Rham complex  with coefficients in  $\cM$:
\[ R^i\cM^\cC \simeq \mathcal H^i(\cC;\cM), \quad i\ge 0.\]

\subsection{A local-to-global spectral sequence} 
One has $I^\cC = \Gamma \circ (-)^\cC$. Moreover, when $\cI$ is an injective object in $\Rep(\cC)$, one has $\mathcal H^i(\cC;\cI)=0$ for $i>0$ \cite{Ugo-derived}. As a result, there is a spectral sequence, converging to $\mathbb H^\bullet(\cC;\cM)$, whose second term is
$$E_2^{pq} = H^p(X,\mathcal H^q(\cC;\cM)).$$ 
 
\subsection{A Hochshild-Serre spectral sequence} \label{HS}
Let us consider an extension of Lie algebroids as in \eqref{extalg}. 
As   we already noticed, $\calL$ is a sheaf of $\cO_X$-Lie algebras, i.e.,
it has a vanishing anchor. Thus, if $\cM$ is a representation of $\cE$, 
the $\calL$-invariant submodule  $\cM^\calL$ is an  $\cO_X$-module, and moreover,
it is a representation of $\cQ$. One has a commutative diagram of functors
\begin{equation}\label{functors}
\xymatrix{
\Rep(\cE) \ar[r]^{(-)^\calL} \ar[rd]_{I^\cE} & \Rep(\cQ) \ar[d]^{I^\cQ} \\
& \Vect}
\end{equation}
The functors $(-)^\calL$ and $I^\cQ$ are left-exact, and moreover, $(-)^\calL$
maps injective objects of $\Rep(\cE)$ to
$I^\cQ$-acyclic objects of $\Rep(\cQ)$  (\cite[Prop.~2.4.6 (vii)]{Ka-Sch},\cite{Ugo-derived}), so that there is a Grothendieck spectral sequence  converging to $R^\bullet I^\cE(\cM)$, whose second 
page is $E_2^{pq} = R^pI^\cQ(R^q\cM^\calL)$ \cite{Tohoku}. This generalizes the Hochschild-Serre spectral sequence one has for extensions of Lie algebras \cite{Hoch-Serre53,BMRT}.
Changing notation, we have \cite{Ugo-derived,BMRT}:

\begin{thm} For every representation $\cM$ of $\cE$ there is a spectral sequence   converging to  $\mathbb H^{\bullet}(\cE;\cM)$, whose second page is 
\begin{equation}\label{2ndpageAlg} E_2^{pq}=\mathbb H^p(\cQ;\mathcal H^q(\calL;\cM)).\end{equation}
\end{thm}

It may be useful to record the explicit form of the five-term sequence of this spectral sequence:
\begin{multline}
0 \to \mathbb H^1(\cQ;\cM ) \to \mathbb H^1(\cE;\cM) \to  \mathbb H^0(\cQ;\mathcal H^1(\calL;\cM)) \\
\to \mathbb H^2(\cQ; \cM ) \to \mathbb H^2(\cE;\cM) \,.
\end{multline} 

 \subsection{The universal enveloping algebroid}  The universal enveloping algebra $\mathfrak U(L)$ of a $(\Bbbk,A)$-Lie-Rinehart algebra $L$
 was defined in \cite{Rinehart63}. It is a $\Bbbk$-algebra equipped with a homomorphism (in fact a monomorphism) $\imath\colon A \to \mathfrak{U}(L)$ of ${\Bbbk}$-algebras and a $\Bbbk$-module morphism $\jmath \colon L \to \mathfrak U(L)$, such that the following relations hold:
 \begin{align}
   \label{ideal1}
   [\jmath(s) ,\jmath (t)] - \jmath ([s,t]) &= 0\,, \quad s,t \in L\,,\\
   \label{ideal2}
   [\jmath(s), \imath (f) ] - \imath(a(s)(f)) &= 0\,, \quad s \in L, f\in A
 \end{align}
(here $a\colon L\to \operatorname{Der}_\Bbbk(A)$ is the anchor morphism). One way to construct $\mathfrak{U}(L)$ is to consider the standard enveloping algebra $U(A\rtimes L)$ of the semi-direct product $\Bbbk$-Lie algebra $A   \rtimes L$ and then mod out the ideal generated by the relation $ f (g,s) -(fg,fs)$. $\mathfrak U(L)$ is an $A$-module via the morphism $\imath$, but observe that due to~\eqref{ideal2} the left and right $A$-module structures are different. (The other relation \eqref{ideal1} simply says $\jmath$ is a morphism of $\Bbbk$-Lie algebras.) We also have a morphism $\varepsilon\colon\mathfrak U(L)\to   \mathfrak U(L)/I = A$ (the augmentation morphism) where $I$ is the ideal generated by $\jmath(L)$. Note that $\varepsilon $ is a morphism of $\mathfrak U(L)$-modules but not of $A$-modules, as
$$ \varepsilon (fs) = a(s)(f)$$
when $ f\in A$, $s\in L$.

The construction of the  universal enveloping algebra $\mathfrak U(L)$ of a $(\Bbbk,A)$-Lie-Rinehart algebra is functorial, and therefore
one can define the universal enveloping algebra $\mathfrak U(\cC)$ of a Lie algebroid $\cC$ by taking the sheaf associated with the presheaf obtained by  applying the previous definition to every $(\Bbbk,\cO_X(U))$-Lie-Rinehart algebra $\cC(U)$, where $U$ runs over the open sets in $X$ \cite{Ugo-derived}.

 \subsection{Lie algebroid hypercohomology and  derivations}  We assume that the Lie algebroid $\cC$ is locally free.
  The category $\Rep(\cC)$ and the category of $\mathfrak U(\cC)$-modules are equivalent,
 and  the functors $\mathbb H^i(\cC;-)$ and
 $\Ext^i_{\mathfrak U(\cC)}(\cO_X,-)$ are isomorphic as functors $\Rep(\cC) \to \Bbbk\hbox{\bf-mod}$ \cite{Ugo-derived} . Let $\cJ$ be the kernel of the augmentation morphism $\mathfrak U(\cC) \to \cO_X$.
 By applying the functor $\Hom_{\mathfrak U(\cC)}(-,\cM)$ to the exact sequence of $\mathfrak U(\cC)$-modules
 \begin{equation}
   \label{eq:I}
   0 \to \cJ \to \mathfrak U(\cC) \to \cO_X \to 0\,.
 \end{equation}
 we obtain the exact sequence
 \[ 0 \to I^\cC(\cM) \to \Gamma(X,\cM) \to \Hom_{\mathfrak U(\cC)}(\cJ,\cM)  \to  
   \mathbb H^1(\cC;\cM)   \to H^1(X,\cM)  
 \]
So every element in $ \mathbb H^1(\cC;\cM) $ which goes to zero in $ H^1(X,\cM) $ (for instance, this will always happen if $X$ is affine)
 is represented by a morphism of $\mathfrak U(\cC)$-modules $\phi\colon \cJ \to \cM$. This in turn induces a morphism $D_\phi\colon \cC\to \cM$ by letting
$$D_\phi(x) = \phi(i(x))$$
where $i$ is the natural inclusion $\cC\to\mathcal I$. This is a  derivation of $\cC$ with values in $\cM$, as one has
\begin{equation}\label{DerLM} D_\phi(\{x,y\})   = \phi(i(\{x,y\}))) = \phi(i(x)i(y)-i(y)i(x)) = x(D_\phi(y))-y(D_\phi(x)).
\end{equation} 
The morphism $\phi\mapsto D_\phi$ establishes indeed an isomorphism
\[\Hom_{\mathfrak U(\cC)}(\cJ,\cM) \simeq \Der(\cC,\cM)\,.\]
Thus the module $\cJ$ (co)represents the functor of derivations
\( \cM \mapsto \Der(\cC,\cM)\,.\)
There is an entirely similar situation for the sheaf of derivations, namely
\[\cHom_{\mathfrak U(\cC)}(\cJ,\cM) \simeq \Dder_{\Bbbk}(\cC,\cM)\,.\]  

\subsection{The truncated complex} \label{truncated}
Under the standing assumption that $\cC$ is locally free, from~\cite{Rinehart63} we obtain that the homological version of the Chevalley-Eilenberg complex of $\cC$ is a flat resolution of $\cO_X$:
\[
  \dotsi \to \mathfrak U(\cC)\otimes_{\cO_X}\Lambda^2\cC \to 
  \mathfrak U(\cC)\otimes_{\cO_X}\cC \to \mathfrak U(\cC) \to
  \cO_X \to 0 \,.
\]
Using~\eqref{eq:I} to slice the above sequence we obtain in turn a flat resolution of the module $\cJ$:
\begin{equation}
  \label{res-I}
  \dotsi \to \mathfrak U(\cC)\otimes_{\cO_X}\Lambda^2\cC \to 
  \mathfrak U(\cC)\otimes_{\cO_X}\cC \to \cJ \to 0\,.
\end{equation}
The above motivates to consider for a Lie algebroid $\cC$ and a $\cC$-module $\cM$, the \emph{sharp} truncation of the Chevalley-Eilenberg complex $\sigma^{\ge 1}\Lambda^\bullet\cC^\ast\otimes \cM$ defined by
\[
  \xymatrix{0 \ar[r] & \cC^\ast\otimes \cM \ar[r] &
  \Lambda^2\cC^\ast \otimes\cM  \ar[r] & \dotsi }
\]
with the term $\Lambda^p\cC^\ast\otimes\cM$ placed in degree $p$. There is an obvious short exact sequence
\[
  \xymatrix{
    0 \ar[r] & \sigma^{\ge 1}\Lambda^\bullet\cC^\ast\otimes \cM \ar[r]
    & \Lambda^\bullet\cC^\ast\otimes \cM \ar[r]
    & \cM [0] \ar[r] &  0\,,}
\]
where $\cM[0]$ is the complexe consisting solely of $\cM$ placed in degree $0$, giving rise to the long exact sequence
\begin{equation}
  \dots \to  \mathbb H^i (\cC; \cM)^{(1)} \to  \mathbb H^i(\cC; \cM) \to H^{i} (X, \cM ) \to \mathbb{H}^{i+1}(\cC, \cM)^{(1)} \to\, \dotsi
  \label{les}
\end{equation}
where we have denoted by $\mathbb H^\bullet(\cC;\cM)^{(1)}$ the hypercohomology of the truncated complex. Note that $\mathbb H^0(\cC;\cM)^{(1)}=0$.

\begin{remark} When $X$ is affine the exact sequence \eqref{les} splits into
$$ 0 \to  \mathbb H^0(\cC; \cM) \to H^{0} (X, \cM ) \to \mathbb{H}^{1}(\cC, \cM)^{(1)} \to \mathbb{H}^{1}(\cC, \cM) \to 0,$$
$$ \mathbb{H}^{i}(\cC, \cM)^{(1)}  \simeq  \mathbb{H}^{i}(\cC, \cM) ,\quad i\ge 2.$$
\label{affine}
\end{remark}

By applying the functor $\Hom_{\mathfrak U(\cC)}( - , \cM)$ to the short exact sequence~\eqref{eq:I}, we also obtain the long exact sequence
\begin{equation}
  \dots \to  \mathbb H^i (\cC; \cM) \to H^{i} (X, \cM ) \to \Ext^i_{\mathfrak U(\cC)}(\cJ,\cM)\to \mathbb{H}^{i+1}(\cC; \cM) \to\, \dotsi
  \label{les-1}
\end{equation}
where we have used that $\Hom_{\mathfrak U(\cC)}( \mathfrak{U}(\cC) , \cM)\simeq \Gamma(X,\cM)$, as well as $R^iI^\cC( \cM) \simeq \Ext^i_{\mathfrak U(\cC)}(\cO_X,\cM)$, combined with the identification with the Lie Algebroid cohomology recalled in sect.~\ref{sec:lie-algebr-cohom}. Comparing the sequences~\eqref{les} and~\eqref{les-1} we obtain the isomorphisms
\begin{equation}
  \label{ext-trunc}
  \Ext^i_{\mathfrak U(\cC)}(\cJ,\cM) \simeq \mathbb{H}^{i+1}(\cC, \cM)^{(1)}\,,\quad i\geq 0\,.
\end{equation}
One easily shows that the derivation functor
$$ \Der(\cC,-) \colon \Rep(\cC) \to \Vect,\qquad \cM \mapsto  \Der(\cC,\cM)$$
is left-exact. As a simple corollary of what was proved in \cite{Ugo-derived}, we have the following
\begin{prop}\label{Der&Trunc}
  The cohomology of the truncated complex is isomorphic (up to a shift) to the derived functors of the derivation functor:
  \[
    R^i\Der(\cC,\cM) \simeq \mathbb{H}^{i+1}(\cC;\cM)^{(1)}\,.
  \]
\end{prop}
\begin{proof}
  In effect, we are deriving $\Der(\cC; - ) \simeq \Hom_{\mathfrak{U}(\cC)} (\cJ, - )$.
\end{proof}

\begin{remark} If we analogously denote by $\mathcal H^\bullet(\cC;\cM)^{(1)}$ the cohomology sheaves of the truncated complex, we have of course
  \[\mathcal H^0(\cC;\cM)^{(1)}=0,\qquad \mathcal H^i(\cC;\cM)^{(1)}=\mathcal H^i(\cC;\cM)\quad\mbox{for}\quad i>1\,,\]
  and
  \[ \mathcal{H}^{i+1}(\cC;\cM)^{(1)} \simeq R^i\Dder_{\Bbbk}(\cC,\cM)\,; \quad\text{in particular,}\  \mathcal{H}^{1}(\cC;\cM)^{(1)} \simeq\Dder_{\Bbbk}(\cC,\cM)\,.\]
\end{remark}

\begin{remark} \label{hyperzero}
The hypercohomology $\mathbb H^\bullet(\cC;\cM)^{(1)}$ can be computed as the cohomology of the total complex of the double
complex $H^{p,q} = \check C^p(\mathfrak U, \sigma^{\ge 1}\Lambda^q\cC^\ast\otimes \cM$), where $\mathfrak U$ is any affine cover of $X$, and $\check C^\bullet$ denotes the associated \v Cech complex.
Since $H^{0,0}=0$, we have $\mathbb H^0(\cC;\cM)^{(1)}=0$, as already noted.  
\end{remark}

\subsection{A long exact sequence for the $\Der$ functor} \label{Derfunct}
Let us consider an extension of Lie algebroids as in \eqref{extalg}. As we already discussed, the center $Z(\calL)$ is
a $\cQ$-module, and hence also an $\cE$-module. Moreover, we give is also the structure of a trivial $\calL$-module.
Note that the forgetful functors $\Rep(\cQ)\to\Rep(\cE)$ and $\Rep(\cE)\to\Rep(\calL)$ are exact, so that they map
injective objects of $\Rep(\cQ)$ to injective objects of $\Rep(\cE)$ and $\Rep(\calL)$. We take an injective resolution $\cI^\bullet$ of $Z(\calL)$ as an object in $\Rep(\cQ)$, and apply the functor $\Der(-,\cI^\bullet)$ to the exact sequence \eqref{extalg}.
We obtain an exact sequence of complexes of vector spaces
$$ 0 \to \Der(\cQ,\cI^\bullet)\to \Der(\cE,\cI^\bullet)\to \Der(\calL,\cI^\bullet) \to 0 $$
and taking the long exact sequence of cohomology we obtain (as $ R^i\Der(\cQ,Z(\calL))  $ $\simeq  
\mathbb H^{i+1}(\cQ;Z(\calL))^{(1)}$)
\begin{multline}\label{Derseq} 
 0 \to \Der(\cQ,Z(\calL)\to \Der(\cE,Z(\calL))\to    \Der(\calL,Z(\calL))  \xrightarrow{\delta} \\  \mathbb H^{2}(\cQ;Z(\calL))^{(1)} \to \mathbb H^{2}(\cE;Z(\calL))^{(1)} \to \dots 
 \end{multline}

\bigskip
\section{Free Lie algebroids} \label{freealg}
\subsection{Free Lie and associative algebras over a set} \label{free}
We start by recalling the construction of the free Lie algebra over a set and it
universal enveloping algebra (\cite{Reutenauer}, see also  \cite{Sternberg-LieAlgebras}).
Let $\Bbbk$ be a field and $A$ a commutative, associative $\Bbbk$-algebra with unit. 
We briefly remind the construction of a free Lie algebra over a set, and its universal
enveloping algebra.  If  $S$ is a set, and $M_S$ the associated magma,
the vector space $A_S=\Bbbk[M_S]$ freely generated by $M_S$ over $\Bbbk$ is 
an algebra over $\Bbbk$, with product given by the product in the magma.

Let $I_S$ be the  two-sided ideal ideal generated in $A_S$  by the elements
$$xx,\quad x\in A_S \quad \mbox{and} \quad (xy)z+(zx)y+(yz)x,\quad x,y,z\in A_S.$$
The quotient $A_S/I_S$ is a $\Bbbk$-Lie algebra --- the free $\Bbbk$-Lie algebra over $S$ ---
that we denote   $\operatorname{Lie}_{\,\Bbbk,S}$. 

Similarly, let $V_S$ be the vector space freely generated by $S$ over $\Bbbk$, and
let $\operatorname{Ass}_{\Bbbk,S}$ be its tensor algebra --- the free associative $\Bbbk$-algebra over the set $S$. 
$\operatorname{Ass}_{\Bbbk,S}$ turns out to be isomorphic to the universal enveloping algebra of 
$\operatorname{Lie}_{\,\Bbbk,S}$ \cite{Reutenauer,Sternberg-LieAlgebras}.

Any Lie algebra $\mathfrak g$ over $\Bbbk$ can be realized as the quotient of a free Lie algebra.
If $S=\{x_i\}$ is a set of generators of $\mathfrak g$, one has indeed a surjection 
$\operatorname{Lie}_{\,\Bbbk,S}\to\mathfrak g$.

\subsection{The free Lie-Rinehart  algebra over a set} \label{freeLR}
Let $\Bbbk$ be a field, and $A$ a commutative associative algebra over $\Bbbk$.
Let $S$ be a set equipped with a map $a_S\colon S \to \Der_\Bbbk (A)$ such the
induced map $\operatorname{Lie}_{\,\Bbbk,S} \to \Der_\Bbbk (A)$ is 
a morphism of $\Bbbk$-Lie algebras (that we denote by the same symbol).\footnote{We could equivalently require
the existence of the map $\operatorname{Lie}_{\,\Bbbk,S} \to \Der_\Bbbk(A)$, as the composition with the canonical
map $S\to \operatorname{Lie}_{\,\Bbbk,S}$ yields the associated map $S\to \Der_\Bbbk(A)$.}
With this data, we can make 
$$L_{A,S} = A \otimes_\Bbbk \operatorname{Lie}_{\,\Bbbk,S} $$
into a $(\Bbbk,A)$-Lie-Rinehart algebra. We call this the {\em free $(\Bbbk,A)$-Lie-Rinehart algebra} over the pair $(S,a_S)$. It has a natural map $S \to L_{A,S}$. 

\begin{prop}\label{mapexists} Let $(L,a_L)$ be a $(\Bbbk,A)$-Lie-Rinehart algebra, and let $f\colon S\to L$ be a map such that the
diagram
$$\xymatrix{
S \ar[r]^f \ar[dr]_{a_S} & L \ar[d]^{a_L} \\ & \Der_\Bbbk(A)}
$$ 
commutes. There is a unique Lie-Rinehart algebra morphism $g\colon  L_{A,S} \to L$ such that the diagram 
$$\xymatrix{
S \ar[rd]^f \ar[d] \\   L_{A,S}  \ar[r]_{g}  & L
}$$ 
commutes. 
\end{prop}
\begin{proof}
By regarding $L$ as a $\Bbbk$-Lie algebra, there is a map $\tilde g\colon \operatorname{Lie}_{\,\Bbbk,S} \to L$ making the diagram
$$\xymatrix{
S \ar[rd]^f \ar[d]_{i_S} \\  \operatorname{Lie}_{\,\Bbbk,S}  \ar[r]_{\tilde g}  & L
}$$ 
commutative. This defines the map $g$ as $g(\alpha\otimes\xi) = \alpha\,\tilde g(\xi)$. The only thing we need to check is the compatibility between the anchors. If we set $\xi=i_S(s)$, we have indeed
$$a_L(g(\alpha\otimes \xi) )= \alpha\,a_L(\tilde g(i_S(s)))
=\alpha\,a_L(f(s)) =\alpha\,a_S(s) = a_{L_{A,S}}(\alpha\otimes\xi).
$$

\end{proof}

The universal enveloping algebra of $L_{A,S}$ can be constructed as follows. 
Let $\operatorname{Ass}_{\Bbbk, A\sqcup S}$ be the free associative $\Bbbk$-algebra generated by $S$ and $A$,\footnote{$A$ is considered as a set.} and 
let $\imath_{A\sqcup S}\colon S \to \operatorname{Ass}_{A,S}$ be the natural map. Let us temporarily denote the image of the elements of $A\sqcup S$ under $\imath_{A\sqcup S}$ by angle brackets: $\langle \alpha\rangle = \imath_{A\sqcup S} (\alpha)$, and $\langle s\rangle = \imath_{A\sqcup S} (s)$. Now, let $J$ be the two-sided ideal in $\operatorname{Ass}_{\Bbbk, A\sqcup S}$ generated by the elements
\begin{equation}\label{J}
  \begin{gathered}
    \langle \alpha\rangle + \langle\beta\rangle - \langle \alpha+\beta\rangle\,,\\
    \langle \alpha\rangle  \langle\beta\rangle - \langle \alpha\,\beta\rangle\,,\\
    \langle s\rangle  \langle\alpha\rangle - \langle \alpha\rangle
    \langle s\rangle - \langle a_S(s)(\alpha)\rangle \,,\\
    k - \langle k\rangle\,,
  \end{gathered}
\end{equation}
for all $\alpha,\beta\in A$, $s\in S$,  and $k\in\Bbbk$, and define
\begin{equation}
  \widetilde{\operatorname{Ass}}_{A,S} = \operatorname{Ass}_{\Bbbk, A\sqcup S}/J\,.
\end{equation}
$\widetilde{\operatorname{Ass}}_{A,S}$ is a $\Bbbk$-algebra\footnote{Observe that the relation $0-\langle 0\rangle$ automatically holds by cancellation.} and the map $\imath_A\colon A\to \widetilde{\operatorname{Ass}}_{A,S}$ defined by sending $\alpha$ to the class of the generator $\langle\alpha\rangle$ is a $\Bbbk$-algebra homomorphism.
\begin{remark}
  One can informally define $\widetilde{\operatorname{Ass}}_{A,S}$ as the $\Bbbk$-algebra generated by $A$ (as a $\Bbbk$-algebra) and $S$ subject to the relation $s \alpha - \alpha s = a_S(s)(\alpha)$. Thus 
  there are certain similarities between $\widetilde{\operatorname{Ass}}_{A,S}$ and the so-called ``skew polynomial rings,'' or formal differential operator algebras, and extensions thereof (see, for instance \cite{MR3468728}). In fact they are the same when $S$ is equal to a singleton.
\end{remark}
\begin{remark}
  There is an alternative---perhaps more conceptual---definition of $\widetilde{\operatorname{Ass}}_{A,S}$, where we realize it as an algebra of differential operators defined in the following way.

  Let $\operatorname{Ass}_{A,S}=A\otimes_\Bbbk \operatorname{Ass}_{\Bbbk,S}$ be the free $A$-algebra generated by $S$, and let $R = \End_k(\operatorname{Ass}_{A,S})$ denote the ring of endomorphism of the underlying $\Bbbk$-module. For all $\alpha,\beta\in A$ and $s\in S$ consider the following endomorphisms in $R$:
  \begin{align*}
    \sigma_s(\alpha\otimes x) &= \alpha\otimes sx + a_S(s)(\alpha)\otimes s\,, \\
    \sigma_\beta(\alpha\otimes x) &= \beta\alpha \otimes x\,,
  \end{align*}
  where $x$ is any word of the generators in $S$, namely $x=s_1\otimes \dots \otimes s_n$, for some $n\in \N$, and $sx$ means $s\otimes x$. It is not difficult to see that the subring of $R$ generated by the above endomorphisms is isomorphic to $\widetilde{\operatorname{Ass}}_{A,S}$ as a $\Bbbk$-algebra, and that the map $\imath_A$ defined above corresponds to the map $\alpha\mapsto \sigma_\alpha$, for all $\alpha\in A$.
\end{remark}
The algebra $\widetilde{\operatorname{Ass}}_{A,S}$ admits two notable filtrations (cf.\ \cite[\S 5.2]{Kapranov-free}). One is recursively defined by $F^n \widetilde{\operatorname{Ass}}_{A,S}=0$ if $n<0$, and
\[
  F^n \widetilde{\operatorname{Ass}}_{A,S} = \Bigl\lbrace u \;\Big\vert\;
  [u, A]\subseteq F^{n-1}\widetilde{\operatorname{Ass}}_{A,S} \Bigr\rbrace.
\]
This is like the standard filtration one would use in rings of differential operators, and it makes  $\widetilde{\operatorname{Ass}}_{A,S}$ a $D$-algebra in the sense of \cite{MR1237825}. The other filtration is based on the number of generators from $S$, namely define $G^n \widetilde{\operatorname{Ass}}_{A,S}$ to be the $\Bbbk$-submodule generated by words in elements of $A$ and $S$ containing up to $n$ elements of $S$, that is
\begin{equation}
  \label{G-filt}
  G^n \widetilde{\operatorname{Ass}}_{A,S}= \Bigl\lbrace \sum \alpha_{i_0}\,s_{i_1} \alpha_{i_1}\dots s_{i_k}\alpha_{i_k} \;\big\vert\; s_{i_j}\in S,\, \alpha_{i_j} \in A \Bigr\rbrace,
\end{equation}
where the sum is over finite subset $\{s_{i_1},\dots ,s_{i_k}\}\subset S,\: k\leq n$.  The number $n$ above is the ``degree'' of $u$, which is well defined thanks to the relation $s\alpha - \alpha s = a_S(s)(\alpha)$.
(We can think of an element $u$ as a noncommutative polynomial in the variables $\lbrace s\in S\rbrace$ with coefficients in $A$.) Then, it is immediately verified that
\begin{equation}
  \label{FG-grad}
  \gr^F_\bullet \widetilde{\operatorname{Ass}}_{A,S} = A [ S ]\,, \;\text{and}\;
  \gr^G_\bullet \widetilde{\operatorname{Ass}}_{A,S} = \operatorname{Ass}_{A,S},
\end{equation}
where on the left we have the free (polynomial) commutative $A$-algebra generated by $S$, whereas on the right we have the free (not necessarily commutative) $A$-algebra generated by $S$.
\begin{remark}
  \label{rem-G-filt}
  Whereas $G^\bullet$ only defines a filtration on $\widetilde{\operatorname{Ass}}_{A,S}$, observe that for each $u\in G^n \widetilde{\operatorname{Ass}}_{A,S}$ its class in $\operatorname{Ass}_{A,S}$ has a representative we can define by choosing a definite ordering, for example by placing all generators from $S$ to the right. Thus $u$ can unambiguously be written as
  \(
  \sum_{k\leq n} u^{(k)}
  \)
  where each term $u^{(k)}$ is inductively defined to have the same formal expression in terms of $A$ and $S$ as each projection $\pi_k(u -\sum_{n\geq l > k} u^{(l)})$ in $\gr^G_k\bigl(\widetilde{\operatorname{Ass}}_{A,S}\bigr)$, $k\leq n$, namely each $u^{(k)}$ is of the form
  \(
  u^{(k)} = \alpha_k \, s_{i_1}\dots s_{i_k}.
  \)
\end{remark}
\begin{prop}
 $\widetilde{\operatorname{Ass}}_{A,S} $ is isomorphic to the universal enveloping algebra
 of the free   $(\Bbbk,A)$-Lie-Rinehart algebra $L_{A,S}$
 over $S$.
 \end{prop}
 \begin{proof}
   To begin with, we recall the universal property of the universal enveloping algebra $\mathfrak{U}(L)$ of a $(\Bbbk,A)$-Lie-Rinehart $(L,a)$ \cite{Moerdijk}. If $B$ is a unital associative $\Bbbk$-algebra, we denote by $B_{\mathrm{Lie}}$ the algebra $B$ regarded as a Lie algebra with the commutator bracket. The universal enveloping algebra $\mathfrak{U}(L)$ solves the following problem: for every $\Bbbk$-algebra homomorphism $i\colon A \to B$, and every morphism of $\Bbbk$-Lie algebras $j\colon L \to B_{\mbox{\tiny Lie}}$ such that for all $\alpha\in A$ and $x\in L$
   \begin{equation}\label{conds}
     i(\alpha)\,j(x) = j(\alpha x)\quad \text{and} \quad [j(x),i(\alpha)] =i( a(x)(\alpha))\,,
   \end{equation}
   there is a unique morphism of $A$-modules $\mathfrak{U}(L) \to B$ such that the diagrams
   \begin{equation}
     \xymatrix{
       L \ar[r]^j \ar[d] & B \\
       \mathfrak{U}(L) \ar[ur] } 
     \qquad
     \xymatrix{
       A \ar[r]^i \ar[d] & B \\
       \mathfrak{U}(L) \ar[ur] } 
   \end{equation}
   commute (here $B$ is an $A$-module via the map $i$).

   By Proposition \ref{mapexists} we have a map $\jmath_L \colon L_{A,S} \to  \widetilde{\operatorname{Ass}}_{A,S}$, which together with $\imath_A\colon A\to \widetilde{\operatorname{Ass}}_{A,S}$ satisfy the relations~\eqref{conds}. Therefore there is a unique morphism $u\colon \mathfrak{U}(L_{A,S})\to \widetilde{\operatorname{Ass}}_{A,S}$. To prove that $u$ is an isomorphism, we show that $\widetilde{\operatorname{Ass}}_{A,S}$ is itself universal with respect to the property just recalled.

   To this end, let $\imath_B\colon A \to B$  be a $\Bbbk$-algebra morphism and let $\jmath_B\colon L_{A,S}\to B_{\mathrm{Lie}}$ be a $\Bbbk$-Lie algebra morphism satisfying~\eqref{conds}. $\jmath_B$ has an underlying set map
   $\jmath_0\colon S\to B$ which we combine into the set map
   \(
     (\imath_B,\jmath_0) \colon A \sqcup S \to B\,.
   \)
   Consider the diagram:
   \begin{equation}
     \xymatrix@C+1pc{
       A\sqcup S \ar[r]^{(\imath_B,\jmath_0)} \ar[d] & B \\
       \operatorname{Ass}_{\Bbbk, A\sqcup S} \ar[ur]_{\tilde v} \ar[d] \\
       \widetilde{\operatorname{Ass}}_{A,S}  \ar@{.>}[uur]_v
     }
   \end{equation}
   The upper solid triangle results from the universal property of the free associative $\Bbbk$-algebra, and it is easily verified that $\tilde v$ vanishes on the relations~\eqref{J}, so $J\subseteq \Kernel \tilde v$, and the map $v$ exists. Evidently, we have $\imath_B=v\circ \imath_A$, and, by restricting to the generator in $S$ and Proposition~\ref{mapexists}, we also see that $\jmath_B = v\circ \jmath_L$, so $\widetilde{\operatorname{Ass}}_{A,S}$ satisfies the universal property, as wanted.
 \end{proof}

 For any Lie-Rinehart algebra $L$, the left augmentation ideal of $\mathfrak{U}(L)$ is $K = \Kernel (\epsilon)$, where $\epsilon\colon \mathfrak{U}(L)\to A$ is defined by $\epsilon(u) = \bar a (u)(1)$, and $\bar a$ is the extension of the anchor map of $L$ to its enveloping algebra. In effect, $\bar a$ is the unique map determined by the universal property of $\mathfrak{U}(L)$ relative to the morphism provided by the anchor map: $a \colon L\to \Der_{\Bbbk}(A)\subset \End_{\Bbbk}(A)$. Note that $\epsilon$ is not a homomorphism, but it satisfies the identity $\epsilon(uv) = \epsilon(u\, \epsilon(v))$, for all $u,v\in \mathfrak{U}(L)$ \cite{Moerdijk}. $K$ is generated by the image of $L$ in $\mathfrak{U}(L)$, and $\mathfrak{U}(L)/K \simeq A$.

 Let us denote by $K_{A,S}$ the augmentation ideal for $L = L_{A,S}$.\footnote{The following result is not used elsewhere in this paper. We record it here for the sake of completeness.}
 
 \begin{prop}
   $K_{A,S}$ is a free $U(L_{A,S} )$-module. As a consequence the cohomology groups \allowbreak $H^i(L_{A,S} ;M)$ vanish for $i \ge 2$ for any representation  $M$ of $L_{A,S}$.
 \label{freevanish}
\end{prop}
 
\begin{proof}
  The ideal $K_{A,S}$ is generated by the image of $L_{A,S}$. Therefore $K_{A,S}$ is generated by $S$. More directly, recall that in $\widetilde{\operatorname{Ass}}_{A,S}$ we can unambiguously write $u$ of degree $n$ as
  \[
    u = 
    \sum_{i : \lvert i\rvert\leq n} s_i \,\alpha_i\,,
  \]
  where $\alpha_i\in A$, $i$ is a multiindex, and $s_i = s_{i_1}\dots s_{i_k}$. Therefore $\epsilon(u)=0$ precisely when $u$ has the form $u = \sum_{s\in S}u_s\,s$, $u_s\in \cramped{\widetilde{\operatorname{Ass}}_{A,S}}$, i.e.\ $u$ belongs to $K_{A,S}$.
  Thus $K_{A,S}$ is the image of the free left  $\widetilde{\operatorname{Ass}}_{A,S}$-module generated by $S$ under the map
  $\sum_{s\in S}u_s\otimes s\mapsto \sum_{s\in S}u_s\,s$.

  Suppose $\sum_{s\in S}u_s\,s=0$. By Remark~\ref{rem-G-filt}, and the considerations on the $G$-filtration~\eqref{G-filt} preceding it, we can assume the terms are ordered as discussed there, namely with all the generators from $S$ to the right, and as consequence we can also assume the $u_s$ to be homogeneous of degree $n$ in the number of generators in $S$, in effect projecting to the graded algebra $\operatorname{Ass}_{A,S}$. We have
  \[
    0 = \sum_{s\in S} u_s \, s = \sum_{s\in S, i: \lvert i\rvert = n} \alpha_s s_{i_1}\dots s_{i_n} \, s \,.
  \]
  Since the rightmost letters are all different, and $\operatorname{Ass}_{A,S}$ is free as an $A$-module, this implies that all $\alpha_s=0$, and so $u_s=0$.
  
  By applying the functor $\Hom_{\mathfrak{U}(L_{A,S})}(-,M)$ to the resolution 
  \begin{equation}  0 \to K_{A,S} \to  \mathfrak{U}(L_{A,S}) \to A \to 0
    \label{shortres}
  \end{equation}
  of $A$ by free $\mathfrak{U}(L_{A,S})$-modules, one gets the second claim; one uses the isomorphism $H^i(L_{A,S}, -) \simeq \Ext^i_{\mathfrak{U}(L_{A,S}) }(A,-)$ as functors $\operatorname{Rep}(L_{A,S})\to \mathbf{Vect}_\Bbbk$ \cite{Rinehart63,Ugo-derived}. 
\end{proof}
 
\begin{thm} \label{freequot} Every $(\Bbbk,A)$-Lie-Rinehart algebra $L$ is a quotient of the universal free $(\Bbbk,A)$-Lie-Rinehart algebra $L_{A,S}$ over some set $S$. \label{univ}
\end{thm}
\begin{proof} If we look at $L$ as a Lie algebra over $\Bbbk$, there is a surjective Lie algebra morphism  $f\colon\operatorname{Lie}_{\,\Bbbk,S}\to L$ for some set $S$.
We define $a_S\colon   \operatorname{Lie}_{\,\Bbbk,S}\to \Der_\Bbbk( A)$ as $a_S = a \circ f$, and use this to define a free Lie algebroid
$L_{A,S}=A\otimes \operatorname{Lie}_{\,\Bbbk,S}$, with a naturally defined map $\tilde f \colon L_{A,S} \to L$. One easily checks that this an $A$-linear morphism which is also a morphism of $\Bbbk$-Lie algebras, and satifies $a\circ \tilde f = a_S \colon L_{A,S}\to \Der_\Bbbk (A)$.
\end{proof}

\begin{remark}  \label{functorial} (A functorial construction) The previous construction of free Lie-Rinheart algebras
has an equivalent description in terms of adjoint functors \cite{casas-ladra-pira,Kapranov-free}.
Let $\mathbf{Vect}_\Bbbk$ and  $\mathbf{Lie}_\Bbbk$ be the categories of $\Bbbk$-vector spaces and 
$\Bbbk$-Lie algebras, respectively. Then the free Lie algebra functor $\mathbf{Flie}_\Bbbk\colon
\mathbf{Vect}_\Bbbk\to \mathbf{Lie}_\Bbbk$ is the left adjoint to the obvious forgetful functor
$\mathbf{Lie}_\Bbbk\to\mathbf{Vect}_\Bbbk$. 
If we consider the categories 
$\mathbf{Vect}_\Bbbk/\operatorname{Der}_\Bbbk(A)$ and $\mathbf{Lie}_\Bbbk/\operatorname{Der}_\Bbbk(A)$ of pairs $(V,b)$, $(L,a)$ respectively, with $b\colon V\to \operatorname{Der}_\Bbbk(A)$ a linear morphism, and $a\colon L\to \operatorname{Der}_\Bbbk(A)$ a morphism of Lie algebras,
we obtain a functor
\begin{equation}\label{functor} \mathbf{Vect}_\Bbbk/\operatorname{Der}_\Bbbk(A) \to \mathbf{Lie}_\Bbbk/\operatorname{Der}_\Bbbk(A)\end{equation}
which is again left adjoint to the obvious forgetful functor. Composing this
with the tensor product functor $A \otimes - $ we get the free Lie-Rinehart functor
\begin{equation}\label{functorLR} \mathbf{FreeLR} \colon \mathbf{Vect}_\Bbbk/\operatorname{Der}_\Bbbk(A) \to \mathbf{LR}_A \end{equation}
(where $ \mathbf{LR}_A$ is the category of $(\Bbbk,A)$-Lie-Rinehart algebras).

Finally, let $\mathbf{Set}/\operatorname{Der}_\Bbbk(A)$ be the category of pairs $(S,a_S)$,
where $S$ is a set,  and $a_S$  map $a_S\colon S \to \Der_\Bbbk (A)$ such that the
induced map $\operatorname{Lie}_{\,\Bbbk,S} \to \Der_\Bbbk (A)$ is 
a morphism of $\Bbbk$-Lie algebras. By taking the free vector space over $S$ this gives a functor
$$\mathbf{Set}/\operatorname{Der}_\Bbbk(A) \to \mathbf{Vect}_\Bbbk/\operatorname{Der}_\Bbbk(A).$$
By composing this functor with the functor \eqref{functorLR} one obtains the functor 
$(S,a_S) \mapsto L_{A,S}$  we implicitly defined  in Section \ref{freeLR}.
\end{remark}

\subsection{Free Lie algebroids}
As all constructions in the previous sections are functorial, they can be sheafified. So, if $\cS$ is a sheaf of sets on a scheme $X$, we can at first construct a sheaf $\Lie_{\,\Bbbk,\cS}$ of
$\Bbbk_X$-Lie algebras, by taking the sheaf associated to the presheaf 
whose space of sections  over an open subset $U\subset X$ is the free Lie algebra over the set $\cS(U)$. Let us assume that 
$a_{\cS}\colon \cS\to\Theta_X$ is a morphism of sheaves of sets such that the induced morphism
$a_{\cS}\colon \Lie_{\,\Bbbk,\cS}\to \Theta_X$ is a morphism of sheaves of $\Bbbk$-algebras. Then the $\cO_X$-module
$\cO_X \otimes_\Bbbk \operatorname{Lie}_{\Bbbk,\cS}$,
using the construction of the transformation Lie-algebroid at the end of Section \ref{BenjiVolodya}, becomes a sheaf of Lie-Rinehart algebras. This defines a sheaf $\calL_{\cS}$ on $X$ which is a  Lie algebroid. We call it the \emph{free  Lie algebroid over $\cS$} (The choice of the scheme $X$ is understood. Moreover, although we do not record the choice of the morphism $a_{\cS}\colon \cS\to\Dder_\Bbbk\cA$ in the notation, we should remember that 
$\calL_{\cS}$ depends on it).

Theorem \ref{univ} immediately implies

\begin{corol} Every  Lie algebroid over $X$ is the quotient of the   
 free  Lie algebroid $\calL_{\cS}$ for some sheaf of sets $\cS$ on $X$
and some morphism of sheaves of sets $a_{\cS}\colon \cS\to\Theta_X$.
\label{freequot-cor}
\end{corol}

Also the functorial construction of Section \ref{functorial} immediately generalizes to Lie algebroids.

\begin{remark}
  If the sheaf of sets $\cS$ is locally constant with a finite stalk, the sheaf of Lie algebras 
$\calL\!\text{\em ie}_{\,\Bbbk,\cS}$ is a locally free $\Bbbk_X$-module of finite rank. As a result,
the free Lie algebroid $\calL_{\cS}$  is a locally free $\cO_X$-module of finite rank. Corollary \ref{freequot-cor} can be strengthened to the claim that every locally free Lie algebroid of finite rank is the quotient of a free Lie algebroid which is a locally free $\cO_X$-module of finite rank.
\end{remark}

\bigskip
\section{Abelian extensions of  Lie algebroids} \label{abelian}
\subsection{The extension problem} Let $\cQ$ be a locally free Lie algebroid on a scheme $X$, and $\calL$ a bundle of Lie algebras over 
$X$, i.e., a locally free Lie algebroid with vanishing anchor.
Fix a morphism 
$$\alpha\colon \cQ \to \Out( \calL),$$
where $\Out(\calL)$ is the sheaf of outer derivations of $\calL$ (note
that $\Out(\calL)$  has a natural structure of Lie algebroid). 
Such a morphism gives the  center $Z(\calL)$ of $\calL$   a $\cQ$-module structure,
so that we may consider the hypercohomology $\mathbb H^\bullet(\cQ;Z(\calL))$ (we drop the dependence on
the morphism $\alpha$ from the notation), and its truncated version $\mathbb H^\bullet(\cQ;Z(\calL))^{(1)}$.

\subsection{Abelian extensions}
Consider an extension of Lie algebroids as in \eqref{extalg}, with $\calL $ abelian. As in Section \ref{Derfunct},
we give $\calL$ a structure of $\cQ$-, $\cE$- and $\calL$-module. Note that $\Der(\calL,\calL) = \End_{\cO_X}(\calL)$,
so that the morphism $\delta$ in \eqref{Derseq} yields a morphism
\begin{equation} \End_{\cO_X}(\calL) \to \mathbb H^2(\cQ;\calL)^{(1)}\,.
\end{equation}
Applying this morphism to the identity of $\calL$ we obtain a  set-theoretic map
(classifying morphism)
\begin{equation}\label{map} \Ext_{\mbox{\tiny LA}} (\cQ;\calL) \to \mathbb H^2(\cQ;\calL)^{(1)}\end{equation}
where $\Ext_{\mbox{\tiny LA}} (\cQ;\calL)$ is the set of equivalence classes of extensions of $\cQ$ by $\calL$ compatible with the morphism $\alpha$, with the usual equivalence relation. 

In the abelian case, the problem of extending a Lie algebroid $\cQ$ by a sheaf of abelian algebras $\calL$  in unobstructed. The   problem 
of classifying the extensions may be cast in the general form of the theory developed in the paper \cite{vanosdol}. There it is shown that extensions are classified by the first derived functor of the functor $\Der(\cQ;-)$ applied to $\calL$. With the identification of these derived 
functor with the shifted functors $\mathbb H^\bullet(\cQ,-)^{(1)}$ (Proposition \ref{Der&Trunc}) we obtain Theorem \ref{class} for the abelian case.

\bigskip
\section{Nonabelian extensions} \label{nonabelian}

\subsection{The obstruction class} 
As we already discussed, in the nonabelian case
the problem of finding an extension of $\cQ$ by $\calL$ inducing a given morphism $\alpha\colon \cQ \to \Dder(Z(\calL))$ is obstructed by a class in the group  $ \mathbb H^3(\cQ;Z(\calL))^{(1)}$.
This obstruction class was already built in \cite{BMRT} using \v Cech resolutions. Here we want to give an more abstract construction, using the formalism so far developed in this paper. What we are going to do is essentially to generalize the treatment in \cite{knopf} to Lie algebroids. 

As we saw, $\cQ$ can be represented as a quotient of a free Lie algebroid, which we denote $\cF$. The epimorphism $\cF \to \cQ$ induces an epimorphism $\mathfrak U(\cF)\to \mathfrak U(\cQ)$. Let 
$\cK$ be the corresponding kernel, so that we have
$$ 0 \to \cK \to \mathfrak U(\cF)\to \mathfrak U(\cQ) \to 0 .$$
Moreover we denote by $\cJ$ the kernel of the augmentation morphism $\mathfrak U(\cF) \to \cO_X$. Note that $\cK$ injects into $\cJ$, and
$\cK\!\!\cJ$ injects to $\cK$. 
If we denote
$$ {\widetilde \cK}\,^i = \cK^i/\cK^{i+1},\qquad\widetilde \cJ\,^i = \cK^i\cJ / \cK^{i+1}\cJ,\quad\text{for}\ i=0,\dots$$
(with $ \cK^0 = \mathfrak U(\cF)$), the sheaves ${\widetilde \cK}\,^i$, $\widetilde \cJ\,^i$ are locally free $\cO_X$-modules
with a $\mathfrak U(\cQ)$-module structure.
The previous injections define morphisms $ {\widetilde \cK}\,^i \to  {\widetilde \cJ}\,^{i-1}$
and  $ {\widetilde \cJ}\,^i \to  {\widetilde \cK}\,^{i}$. Moreover, there is a morphism $\widetilde  \cJ\,^0 \to \mathfrak U(\cQ)$.

Let $\mathcal X$ be a sheaf of sets, and $a_\cF\colon\mathcal X \to \Dder_\Bbbk(\cO_X)$ a morphism such that the associated free Lie algebroid is isomorphic to $\cF$. Analogously, let   $\mathcal Y$ be a sheaf of sets whose associated sheaf of free Lie algebras is isomorphic to $\cT$. 
Then $\mathcal Y$ generates $\cK$ as a sheaf of free $\mathfrak U(\cQ)$-algebras. Moreover, products of $i$ sections of $\mathcal Y$ mod
$\cK^{i+1}$ generate  ${\widetilde \cK}\,^i$, and products of $i$ copies  of $\mathcal Y$ times sections of $\mathcal X$ mod $\cK^{i+1}\!\!\cJ$
generate $\widetilde\cJ\,^i$.
 
\begin{lemma}
The sequence
\begin{equation}\label{freeres} \dots \to \widetilde \cK\,^2\to \widetilde \cJ\,^1 \to \widetilde \cK\,^1\to\widetilde  \cJ\,^0 \to \mathfrak U(\cQ)
\to \cO_X \to 0 .
\end{equation}
is  a resolution of $\cO_X$ by locally  free $\mathfrak U(\cQ)$-modules.
\end{lemma}
\begin{proof} Brute force diagram chasing. \end{proof}
We can again slice this exact sequence into
\begin{gather} 
 0 \to \cJ \to \mathfrak U(\cQ) \to \cO_X \to 0 \,, \\
\dots \to \widetilde \cK\,^2\to \widetilde \cJ\,^1 \to \widetilde \cK\,^1\to\widetilde  \cJ\,^0 \to  \cJ \to 0 \,. \label{slice2}
\end{gather}

Moreover, we pick a lift $\tilde\alpha\colon\cF\to\Dder (\calL)$ of $\alpha$, so that we have a commutative diagram
\begin{equation}\xymatrix{
& 0 \ar[r] &  \cT \ar[r] \ar[d]^\beta & \cF \ar[r] \ar[d]^{\tilde \alpha} & \cQ \ar[r] \ar[d]^\alpha  & 0 \\
0   \ar[r]  & Z(\calL) \ar[r] & \calL  \ar[r]  \ar[r]^{\operatorname{ad}\ \ \ \  } & \Dder(\calL )\ar[r] & 
\Out(\calL) \ar[r]  & 0 }
\label{diagknopf}\end{equation}
where $\beta$ is the induced morphism.

We define a morphism
\begin{equation}o \colon \widetilde \cJ\,^1 \to Z(\calL).\label{defo}\end{equation}
It is enough to define $o$ on an element of the type $yx$, where $x$ is a generator of $\cF$, and
$y$ is a generator of $\cT$. We let
$$o(yx) = \beta(\{x,y\} )- \tilde\alpha(x) (\beta(y)).$$
Note that if $\ell$ is section of $\calL$, then
$$ \{\beta(\{x,y\} ),\ell\} = \{\{\beta(x),\beta(y)\},\ell\} = \{\operatorname{ad}(\beta(x))(\beta(y)),\ell\} =   \{\tilde\alpha(x) (\beta(y)),\ell\} $$
so that $o$ takes values in $Z(\calL)$. 

We apply the functor $\Hom_{\mathfrak U(\cQ)}(-,Z(\calL))$ to the resolution \eqref{slice2}, obtaining, 
\begin{multline} 0 \to \Der(\cQ,Z(\calL))  \to
\Hom_{\mathfrak U(\cQ)}(\widetilde\cJ\,^0,Z(\calL))   \xrightarrow{d_1} 
\Hom_{\mathfrak U(\cQ)}(\widetilde\cK\,^1 ,Z(\calL))\\ \xrightarrow{d_2 } 
\Hom_{\mathfrak U(\cQ)}(\widetilde\cJ\,^1 ,Z(\calL)) \xrightarrow{d_3 } 
\Hom_{\mathfrak U(\cQ)}(\widetilde\cK\,^2 ,Z(\calL)) 
\to \dots 
\label{resUL}
\end{multline}
By  Section \ref{truncated}, the cohomology of this complex is isomorphic to $\mathbb H^{\bullet+1}(\cQ;Z(\calL))$.
Note also that $o$ is an element in $\Hom_{\mathfrak U(\cQ)}(\widetilde\cJ\,^1 ,Z(\calL))$.

\begin{lemma} $d_3(o) = 0$. Moreover, the cohomology class of $o$ in $\mathbb H^3(\cQ;Z(\calL))^{(1)}$ only depends on $\alpha$.
\end{lemma}
\begin{proof}  To prove that $d_3(o) = 0$ we need to show that $o(yx)$, as in equation \eqref{defo}, is zero when both $y$ and $x$ are
sections of $\cT$. But this follows from the commutativity of the diagram \eqref{diagknopf} (that is, from the definition of $\beta$). To prove that  the cohomology class of $o$ only depends on $\alpha$ means to show that this cohomology class vanishes when $\alpha=0$. In this case,
$\tilde\alpha$ takes values in the inner derivations, i.e., there is a morphism $\alpha\colon \hat\cF\to\calL$ such that $\tilde\alpha(x)(y)=\{\hat\alpha(x),y\}$.
Moreover, $\hat\alpha_{\vert \cT} = \beta$, so that 
$$o(yx) = \hat\alpha(\{x,y\} )- \{\hat\alpha(x),\hat\alpha(y)\} = 0.$$
\end{proof}


\begin{defin} We denote by $\mathbf{ob}(\alpha)$ the cohomology class induced in  $\mathbb H^3(\cQ;Z(\calL))^{(1)}$
by $o$, and call it the {\em obstruction class} associated with $\alpha$.
\end{defin}

\begin{thm} Given a Lie algebroid $\cQ$, a bundle $\calL$ of Lie algebras over $\cO_X$,
and a morphism $\alpha\colon \cQ \to \Out( \calL)$, an extension of Lie algebroids
as in \eqref{extalg}
inducing on $Z(\calL)$ the $\cQ$-module structure given by $\alpha$ exists if and only if
 $\mathbf{ob}(\alpha)=0$.
 \label{obs}
 \end{thm}

\begin{proof}  Assume that an extension as in \eqref{extalg} exists, inducing the given morphism $\alpha$. Write $\cQ$ as the quotient of a free algebroid $\cF$, and lift the morphism $\cE\to\cQ$ to $\cF$, obtaining a commutative diagram
$$\xymatrix{
0 \ar[r] & \cT \ar[r]\ar[d]^\beta &  \cF \ar[r]\ar[d]^\gamma & \cQ \ar[r]\ar@{=}[d] & 0 \\
0  \ar[r] & \calL  \ar[r] & \cE  \ar[r] & \cQ  \ar[r] & 0 }
$$
where $\beta$ is the induced morphism. Define $\tilde\alpha\colon\cF\to \Dder(\calL,\calL)$ by letting
$\tilde\alpha =-\operatorname{ad}\circ\,\gamma$. Then $\tilde\alpha$ is a lift of $\alpha$, 
and for all sections $t$ of $\cT$ and $x$ of $\cF$ one has
\begin{equation}\label{vanishes} \beta(\{x,t\} ) -  \tilde\alpha(x)(\beta(t)) = 0 \end{equation}
so that the obstruction class $\mathbf{ob}(\alpha)$ vanishes.

Conversely, assume that $\mathbf{ob}(\alpha)=0$, and take a lift $\tilde\alpha\colon \cF\to\Dder(\calL,\calL)$.
The corresponding cocycle lies in the image of the morphism $d_2$ in \eqref{resUL},
so it defines a morphism $\beta\colon\cT \to \calL$, which satisfies the equation \eqref{vanishes}. 
Again, we consider the extension 
$$ 0 \to \cT \to \cF \to \cQ \to 0.$$
Note that $\calL$ is an $\cF$-module via $\cF\to\cQ$. The semidirect product $\calL \rtimes \cF$ contains the Lie algebra bundle
$$ \mathcal H = \{(\ell,x)\,\vert\, x\in \cT, \ \ell = \beta(x) \}.$$
The quotient $\cE =  \calL \rtimes \cF /  \mathcal H $ provides the desired extension.
\end{proof}

\subsection{Classifying extensions}
We assume now that the obstruction class $\mathbf{ob}(\alpha)=0$ is zero,
so that the set $\Ext_{\mbox{\tiny LA}}(\cQ,\calL)$ of equivalence classes of extensions of $\cQ$ by $\calL$ is not empty.
We want to show that $\Ext_{\mbox{\tiny LA}}(\cQ,\calL)$  is a torsor on the group $ \mathbb H^2(\cQ;Z(\calL))^{(1)}$.
The idea is to  reduce the problem to the abelian case. We shall be inspired by the treatment in \cite{Mori-ext} for the case
of Lie algebras (actually, this is in turn an adaption to the case of Lie algebras of what was done by Eilenberg and Maclane for
groups \cite{eilenberg-maclane-II}, and the Eilenberg-Maclane paper is much easier to read).

In particular, we shall prove the following result. For clarity, for every morphism $\alpha\colon\cQ\to\Out(\calL)$ we denote
by $\alpha_0$ the induced morphism $\alpha_0\colon\cQ\to \Dder(Z(\calL))$.

\begin{prop} The equivalence classes of extensions of $\cQ$ by $\calL$ inducing $\alpha$ are in a one-to-one correspondence
with equivalence classes of extensions of $\cQ$ by $Z(\calL)$ inducing $\alpha_0$, and are therefore in a one-to-one correspondence
with  the elements of the group $\mathbb H^2(\cQ;Z(\calL))^{(1)}$. 
 \label{abext}
\end{prop} 

To prove Proposition \ref{abext} we need to develop some machinery.
Let $\cC_1$, $\cC_2$ be two Lie algebroids, and assume there are two surjective morphism $f_i\colon \cC_i\to\cQ$. The fibre product
$\cC_1\times_\cQ\cC_2$ has a natural structure of Lie algebroid. Assuming that $\ker f_1$ and $\ker f_2$ have
  isomorphic centres, which we denote  $\cZ$, we define a product   $\cC_1\star\cC_2$ by letting
$$ \cC_1\star\cC_2 = \cC_1\times_\cQ\cC_2/\cZ,$$
where $\cZ$ is mapped to $\cC_1\times_\cQ\cC_2$  as $z \mapsto (z,-z)$. 

Moreover, we shall consider pairs $(\cK,\alpha)$, where $\cK$ is a bundle of $\cO_X$-Lie algebras on $X$, 
whose centre is isomorphic to a fixed bundle of abelian $\cO_X$-Lie algebras $\cZ$, and $\alpha$ is a morphism $\cQ\to\Out(\cK)$.
 We assume that  the obstruction class $\mathbf{ob}(\alpha_0)$ vanishes, so that  for every pair $(\cK,\alpha)$ there are extensions 
$$ 0 \to \cK \to \cE \to \cQ \to 0$$
such that the induced morphism $\cQ\to \Out(\cK)$ coincides with $\alpha$ --- i.e., every pair $(\cK,\alpha)$ is extendible.

Given two pairs $(\cK',\alpha')$, $(\cK'',\alpha'')$, we define their product
$$ (\cK',\alpha') \star (\cK'',\alpha'') = (\cK'\oplus\cK''/\cZ, \alpha'\star\alpha'')$$
where $\cZ$ is embedded as above, and  $\alpha'\star\alpha''\colon\cQ\to \Out(\cK'\oplus\cK')$ is the sum of
$\alpha'$ and $\alpha''$, which acts on the image of $\cZ$ in $\cK'\oplus\cK''$ as the  derivation $(\alpha_0,-\alpha_0)$.
It is easy to check that if $\cE'$ and $\cE''$ are extensions of $\cQ$ by $(\cK',\alpha')$, $(\cK'',\alpha'')$, respectively,
then $\cE'\star\cE''$ is an extension of $\cQ$ by $ (\cK',\alpha') \star (\cK'',\alpha'')$. Since this is compatible with
equivalence, in particular we have a map
\begin{equation}\label{mapext}
\Ext^1(\cQ,\cK') \times \Ext^1(\cQ,\cK'') \to \Ext^1(\cQ,\cK'\star\cK'').
\end{equation}

Note that a derivation of a Lie algebra bundle always restricts to a derivation of its centre.

\begin{lemma} Given pairs $(\cK,\alpha)$ and $(\cZ,\beta)$, where $\beta$ is the restriction of $\alpha$ to $\cZ$,
one has $(\cK,\alpha)\star(\cZ,\beta)\simeq (\cK,\alpha)$.
\end{lemma}
\begin{proof} Direct computation.
\end{proof}

\begin{remark} In this case, the map \eqref{mapext} becomes 
$$\Ext^1(\cQ,\cK) \times \Ext^1(\cQ,\cZ) \to \Ext^1(\cQ,\cK)$$
(i.e., it is the Baer sum, see e.g.~\cite{Hilton-Stammbach}),
and on representing cocycles it is expressed by the sum of cocycles.
\end{remark}

We now fix a reference point in $\Ext_{\mbox{\tiny LA}}(\cQ,\calL)$, that is, we fix an extension $\cE$ of $\cQ$ by $\calL$.
The following two Lemmas provide a proof of Proposition \ref{abext}. 

\begin{prop} Any extension  $\cE'$ of $\cQ$ by $\calL$  is equivalent to a product  $\cE\star\calD$ of $\cE$ 
by an extension $\calD$  of $\cQ$ by $Z(\calL)$.\label{prop1}
\end{prop}

\begin{proof} Choosing an open affine covering $\mathfrak U= \{U_i\}$ over which all bundles $\calL$, $\cE$ and $\cQ$ trivialize,
we may fix local splittings $s_i\colon \cQ_{\vert U_i} \to \cE_{\vert U_i}$. Then we can associate with $\cE$ a triple of \v Cech cochains
$$ \{\alpha_i\} \in\check C^0(\mathfrak U,\cQ^\ast\otimes\Dder(\calL,\calL)), \quad \{\rho_i\}\in\check C^0(\mathcal U,\Lambda^2\cQ^\ast\otimes \calL),\quad \{\phi_{ij}\} \in\check C^1(\mathfrak U,\cQ^\ast\otimes\calL)$$
by letting, for all sections $x$, $y$ of $\cQ$ and $\ell$ of $\calL$,
$$\alpha_i(x)(\ell) = \{s_i(x),\ell\},\quad \rho (x,y) = \{s_i(x),s_i(y)\}-s_i(\{x,y\}), \quad \phi_{ij} = {s_i}_{\vert U_i\cap U_j} -
{s_j}_{\vert U_i\cap U_j}.$$
The $\alpha_i$ are local lifts of the morphism $\alpha\colon\cQ\to \Out(Z(\calL))$,
and satisfy the conditions
\begin{equation}\label{alphacond}  [\alpha_i(x),\alpha_i(y)]-\alpha_i(\{x,y\})=  \operatorname{ad}(\rho_i(x,y)),
\qquad \alpha_i-\alpha_j= \operatorname{ad}\phi_{ij}
.\end{equation}
 Moreover,
 $\{\phi_{ij}\} $ is a cocycle  which describes
$\cE$ as an extension of vector bundles. The cochain $ \{\rho_i\}$ is closed under the Lie-Rinehart differential
\begin{equation}\label{LRclosed}  d_{\alpha_i}\rho_i(x,y,z) = \alpha_i(x)(\rho_i(y,z))-\rho_i(\{x,y\},z) + \mbox{cycl.~perm.} = 0
\end{equation}
and describes $\cE(U_i)$ as a Lie-Rinehart extension
of $\cQ(U_i)$ by $\calL(U_i)$.
Finally, these cochains satisfy a compatibility condition for the Lie-Rinehart algebra structures on $U_i\cap U_j$:
$$(\delta\rho)_{ij} =  d_{\alpha_i} \phi_{ij}.$$

 If $(\{\alpha'_i\}, \{\rho'_i\}, \{\phi'_{ij}\})$ is a triple describing the extension $\cE'$, 
 one can modify  $\{\phi'_{ij}\}$ by adding a coboundary so that $\alpha'_i=\alpha_i$. If we define
\begin{equation} \psi_{ij} = \phi'_{ij}-\phi_{ij} \label{cocycle1Z} \end{equation}
then for all sections $x$ of $\cQ$ and $\ell$ of $\calL$ one has
$\{ \psi_{ij}(x),\ell) \}= 0 $,
so that 
\begin{equation}
\{\psi_{ij}\} \in\check C^1(\mathfrak U,\cQ^\ast\otimes Z(\calL)).
\end{equation}
We obtain therefore an extension of vector bundles
\begin{equation} 0 \to Z(\calL) \to \calD \to \cQ \to 0.\label{Dext}\end{equation}
Finally, we define a 2-cochain $\{\sigma_i\}\in\check C^0(\mathfrak U,\Lambda^2\cQ^\ast\otimes\calL)$
\begin{equation}
\sigma_i = \rho'_i-\rho_i .\label{cocycle2Z}
\end{equation}
Since both $\rho$ and $\rho'$ satisfy the first condition in \eqref{alphacond}, 
$\{\sigma_i\}$ takes values in the centre $Z(\calL)$, and moreover it satisfies
the conditions
\begin{equation} d_{\alpha_i}\sigma_i=0,\qquad (\delta\sigma)_{ij} = d_{\alpha_i}\psi_{ij}.\end{equation}
So the triple $( \{\alpha_i\},\, \{\sigma_i\},\, \{\psi_{ij}\}) $ gives $\calD$ a Lie algebroid structure. 

Equations \eqref{cocycle1Z} and     \eqref{cocycle2Z} express the fact 
 that $\cE\star\calD\simeq\cE'$.
\end{proof}

\begin{prop} Given an extensions $\cE$ of  $\cQ$ by $\calL$ and two extensions $\calD_1$, $\calD_2$ of $\cQ$ by $Z(\calL)$, 
the extensions $\cE_1=\cE\star\calD_1$ and $\cE_2=\cE\star\calD_2$ are equivalent if and only if $\calD_1$ and $\calD_2$ are equivalent.
\label{prop2}
\end{prop}

\begin{proof} If $\calD_1$ and $\calD_2$ are equivalent, then $\cE_1$ and $\cE_2$ are certainly equivalent. Let us prove the converse.
One has a Lie algebroid morphism $f\colon \cE_1\to\cE_2$ such that the diagram
$$\xymatrix{
&& \cE_1 \ar[dr]\ar[dd]^f \\
0 \ar[r] & \calL \ar[ur]\ar[dr] & & \cQ \ar[r] & 0 \\
&& \cE_2 \ar[ru]
}
$$
commutes. If  $\{\phi_{ij}\}$ is a cocycle representing the extension class of $\cE$, and $\{\psi^{(1)}_{ij}\}$ and $\{\psi^{(2)}_{ij}\}$ 
are cocycles representing the extension classes of $\calD_1$ and $\calD_2$ respectively, then, as $\cE_1$ and $\cE_2$ are isomorphic as vector bundles,
$$ \phi_{ij} + \psi^{(1)}_{ij} = \phi_{ij} + \psi^{(2)}_{ij} + \chi_i-\chi_j$$
for some 0-cocycle $\chi$. So   the classes of $\{\psi^{(1)}_{ij}\}$ and $\{\psi^{(2)}_{ij}\}$ in $\Ext^1(\cQ,Z(\calL))$ coincide, i.e, $\calD_1$ and $\calD_2$ 
are equivalent as vector bundle extensions. Then we identify $\calD_1$ and $\calD_2$ 
 as vector bundles. 

We can introduce local splittings $\{s^{1}_i\}$, $\{s^{2}_i\}$ for $\cE_1$ and $\cE_2$ with the corresponding representing triples.
We can again redefine the cocycle (say) $\{\phi^{(2)}_{ij}\}$
so that $\alpha_i^{(1)}=\alpha^{(2)}_i$ (and we shall denote this $\alpha_i$).  We also introduce $\{b_i\}\in \check C^0(\mathfrak U,\cQ^\ast\otimes\calL)$ by letting $f\circ s^{(1)}_i = b_i + s^{(2)}_i$. 
As
$$\alpha_i(x)(\ell) = f ( \alpha_i(x)(\ell)) = \{f(s^{(1)}(x)),\ell\} = \alpha_i(x)(\ell) + \{b_i(x),\ell\},$$
$\{b_i\}$ actually has values in $Z(\calL)$. The Lie algebroids $\calD_1$, $\calD_2$ are represented by the triples
$$(\{\alpha_i\}, \ \{\sigma^{(j)}_i =  \rho ^{(j)}_i -\rho_i\}, \ \{\psi^{(j)}_{ij}\}), \quad j=1,2.$$

Moreover, one has the equalities
\begin{eqnarray} f(\rho^{(1)}_i(x,y)) &=& \rho^{(2)}_i(x,y) + \{s^{(2)}(x),b_i(y)\} - \{s^{(2)}(y),b_i(x)\} - b_i(\{x,y\}) \\
&=& \rho_i(x,y) + \sigma^{(2)}_i(x,y) + (d_{\alpha_i} b_i)(x,y) \\
   f(\rho^{(1)}_i(x,y)) &=&  \rho_i(x,y) + \sigma^{(1)}_i(x,y)\end{eqnarray}
so that 
$$  \sigma^{(2)}_i= \sigma^{(1)}_i - d_{\alpha_i} b_i .$$
Therefore, $\calD_1$ and $\calD_2$ are equivalent.
\end{proof}

\noindent {\em Proof of Proposition \ref{abext}}.  After fixing an extension $\cE_0$, given any  other extension $\cE$ we can realize
it as $\cE_0\star\calD$; the extension $\calD$  gives an element in $\mathbb H^2(\cQ;Z(\calL))^{(1)}$, which as a consequence of Proposition \ref{prop2}
only depends on the equivalence class of $\cE$.  The resulting map $\Ext_{\mbox{\tiny LA}}(\cQ,\calL)\to\mathbb H^2(\cQ;Z(\calL))^{(1)}$
is bijective because it is so in the abelian case. \qed

\medskip
It is clear that  we have proved that  $\Ext_{\mbox{\tiny LA}}(\cQ,\calL)$ is a torsor over $\mathbb H^2(\cQ;Z(\calL))^{(1)}$.
This completes the proof of Theorem \ref{class} in the nonabelian case.

\bigskip
\frenchspacing

\def\cprime{$'$} \def\cprime{$'$} \def\cprime{$'$} \def\cprime{$'$}

\end{document}